\documentclass{amsart}
\usepackage[dvips]{graphicx}
\usepackage{amscd}
\usepackage{pstricks}

\theoremstyle{plain}
\newtheorem{theorem}{Theorem}
\newtheorem{proposition}{Proposition}
\newtheorem{lemma}{Lemma}
\newtheorem{definition}{Definition}
\newtheorem{corollary}{Corollary}

\theoremstyle{remark}

\numberwithin{equation}{section}

\usepackage{amsmath}
\usepackage{amssymb}
\usepackage{graphics}

\newcommand{\R}{\mathbb R}

\newcommand{\p}{\partial}

\newcommand{\la}{\lambda}

\newcommand{\kb}{\overline{k}}

\newcommand{\T}{\mathcal{T}}
\newcommand{\ga}{\gamma}
\newcommand{\wF}{\widetilde{F}}
\newcommand{\wH}{\widetilde{H}}

\begin{document}

\title[]{ A new class of Nilpotent Jacobians in any dimension}
\author[\'A. Casta\~neda]{\'Alvaro Casta\~neda}
\author[A. van den Essen]{Arno van den Essen}
\address[\'A. Casta\~neda]{Departamento de Matem\'aticas, Facultad de Ciencias, Universidad de
  Chile, Casilla 653, Santiago, Chile}
\address[A. van den Essen]{Department of Mathematics, Radboud University Nijmegen, Postbus 9010, 6500 GL Nijmegen,
The Netherlands}
\email{castaneda@uchile.cl, essen@math.ru.nl}
\date{\today}
\thanks{The first author have been partially supported by FONDECYT Regular 1170968.}
\subjclass[2010]{Primary: 14R15 14R10.}
\keywords{Jacobian Conjecture; Nilpotent Jacobian matrix; Polynomial maps}
\begin{abstract}
The classification of the nilpotent Jacobians with some structure has been an object of study because of its relationship with the Jacobian conjecture.
In this paper we classify the polynomial maps in dimension $n$ of the form $H = (u(x,y), u_2(x,y,x_3), \ldots, u_{n-1}(x,y,x_n), h(x,y))$ with $JH$ nilpotent.
In addition we prove that the maps $X + H$ are invertible, which shows that for this kind of maps the Jacobian conjecture is verified.
\end{abstract}

\maketitle

\section{Introduction}

Let $k$ be a field of characteristic zero and $k[X]=k[x_1,\cdots,x_n]$ the polynomial ring in $n$ variables over $k$. Since the remarkable works of  H. Bass {\it{et al.}} \cite{BCW} and A.V. Yagzhev \cite{Y} concerning the Jacobian Conjecture, the study of polynomial maps $H: k^n \to k^n$ such that its Jacobian matrix $JH$ is nilpotent has grabbed the attention of many authors. Although the previously mentioned works establish that, in order to study the conjecture, it is sufficient to focus on maps of the form $X + H$ where $ H $ is homogeneous of degree $ 3,$ the classification of maps with nilpotent Jacobian of any degree, even inhomogeneous, has an interest which goes beyond the Jacobian Conjecture. For example
it led various authors to formulate the following problem:
\medskip

\noindent {\bf{(Homogeneous) Dependence Problem.}} Let $H=(H_1, \ldots, H_n) \in k[X]^n$ (homogeneous of degree $d \geq 1$) such that $JH$ is nilpotent and $H(0) = 0.$
Does it follow that $H_1, \ldots, H_n$ are linearly dependent over $ k $ or equivalently does it follow that the rows of $JH$
are linearly dependent over $k$?

\medskip

An affirmative answer was given in the following cases: rank $JH\leq 1$ in \cite{BCW}, hence if $n=2$ and in case $H$ is homogeneous
of degree $3$  when $n=3$ by D. Wright in \cite{Wright} (resp. when $n=4$ by E. Hubbers in \cite{Hubbers}). In dimension three
an affirmative answer to the homogeneous dependence problem (in any degree) was given by  M. de Bondt and A. van den Essen in \cite{Bondt1}.  On the other hand M. de Bondt in \cite{Bondt2} constructed homogeneous examples in all dimensions $\geq 5$
of nilpotent Jacobians with over $k$  linearly {\em  independent} rows.

\medskip

Although the answer to the dependence problem turned out to be negative in general, studying this problem payed off in several
ways. For example the assumption that the answer to the dependence problem would be positive led the authors in
\cite{EH} to construct a large class of polynomial maps $H$ such that $JH$ is nilpotent. Several of these examples were
subsequently used to find counterexamples to various conjectures, such as Meisters' Cubic Linearization Conjecture \cite{Ess}, the DMZ-Conjecture  \cite{EH2}, the long standing
Markus-Yamabe Conjecture  and the Discrete Markus Yamabe Problem \cite{CEGHM}.

The first negative answer to the dependence problem was found by the second author in \cite{E}, namely
$$H=(y-x^2,z+2x(y-x^2),-(y-x^2)^2).$$
Remarkably, searching for more negative examples in dimension three, the authors of \cite{ChE} showed that,
looking for such examples of the form
$$(u(x,y),v(x,y,z),h(u(x,u),v(x,y,z))$$
the above example is, apart from a linear coordinate change, essentially the only example. This example was generalized
in Proposition 7.1.9 \cite{vE} to give nilpotent Jacobians in all dimensions, with over $k$ linearly independent rows. It was shown
in \cite{CG} that for these examples $H$ and each $\la<0$, the corresponding dynamical system $\dot{x}=F(x)$, where
$F(x)=\la x+H(x)$, has orbits which escape to infinity, hence are counterexamples to the Markus-Yamabe Conjecture.

Recently, in \cite{Yan}, Dan Yan completely classified all $H$ of the form
$$(u(x,y),v(x,y,z),h(x,y))$$
 with nilpotent
Jacobian and over $k$ linearly independent rows. Again they all turned out to be linearly equivalent to the first example
found by the second author.  These results confirm a conjecture of the first author which asserts that if $JH$ is nilpotent,
with over $\R$ linearly independent rows, then the corresponding dynamical system $\dot{x}=F(x)$, where $F(x)=\la x+H(x)$
and $\la<0$, has orbits which escape to infinity. To get more evidence for this last conjecture it is therefore natural to
look for nilpotent Jacobians in dimensions $\geq 4$.

In this paper we pursue this idea and generalize the recent result of Dan Yan to all dimensions $n\geq 3$. More
precisely, we study maps of the form
$$H=(u(x,y),u_2(x,y,x_3),u_3(x,y,x_4),\cdots, u_{n-1}(x,y,x_n),h(x,y))$$
The main result of this paper, Theorem \ref{main1},  completely classifies all such $H$, which Jacobian is nilpotent. Moreover,
in the last section we give a very detailed description  of these maps. This enables us to show that  the corresponding
maps $F=X+H$, which Jacobian determinant equal $1$, are invertible. So we confirm the Jacobian Conjecture
for these maps. A priori, from the construction of the $H$'s it is not at all obvious why $F$ should be invertible. The delicate
proof we give below is, in our opinion, a strong indication that the Jacobian Conjecture might be true
after all (inspite of several  statements of the second author in the past). More evidence in favor of the Jacobian Conjecture
can be found in the works of Zhao and his co-authors, in which the Jacobian Conjecture is firmly embedded
in the framework of Mathieu-Zhao spaces (see \cite{IC}, \cite{GIC}, \cite{MS}, \cite{EWZ2} and \cite{DEZ}).

\section{The nilpotency of $JH$}

In this section we establish a characterization of the nilpotency of $JH$ with $H$ a polynomial map of the form $H = (u(x,y), u_2(x,y,x_3), \ldots, u_{n-1}(x,y,x_n), h(x,y)).$

\begin{proposition}
\label{nilpotency}
  $JH$ nilpotent if and only if
$$u_x+{u_2}_y=0$$
$$u_x{u_2}_y-u_y{u_2}_x-{u_2}_{x_3}{u_3}_y=0$$
$${u_2}_{x_3}(u_x{u_3}_y-u_y{u_3}_x-{u_3}_{x_4}{u_4}_y)=0$$
$${u_2}_{x_3}{u_3}_{x_4}(u_x{u_4}_y-u_y{u_4}_x-{u_4}_{x_5}{u_5}_y)=0$$
$$\cdots$$
$${u_2}_{x_3}{u_3}_{x_4} \cdots {u_{n-1}}_{x_n}(u_xh_y-u_yh_x)=0$$
\end{proposition}

\noindent{\bf Proof (started):} Let $S$ be a new variable and put $T:=S^{-1}$. Then $JH$ is nilpotent if and only if $-JH$ is nilpotent
if and only if $det\,(SI_n+JH)=S^n$ if and only if $d(T):=det\, (I_n+TJH)=1$. Since $d(T)$ is a polynomial in $k[x,y,\cdots,x_n][T]$
of degree $n$ in $T$ and $d(0)=1$, the statement that $d(T)=1$ is equivalent to the fact that for each $1\leq i\leq n$
the coefficient of $T^i$ in $d(T)$ is equal to zero. We will show that the coefficient of $T^1$  being zero gives the first equation,
the coefficient of $T^2$ the second and so on. We use some linear algebra to see this. Therefore put $D_n:=I_n+T JH$. For $1\leq k\leq n$ denote by ${D_n}_{(k)}$ the
$k$-th column of $D_n$. Then
$${D_n}_{(1)}=T\left(\begin{array}{cc}
{u_1}_x\\
\vdots\\
{u_n}_x
\end{array}
\right)+e_1,\,\,
{D_n}_{(2)}=T\left(\begin{array}{cc}
{u_1}_y\\
\vdots\\
{u_n}_y
\end{array}
\right)+e_2,
$$
\noindent and
$${D_n}_{(k)}=e_k+T {u_{k-1}}_{x_k} e_{k-1}, \mbox{ for all } 3\leq k\leq n$$
\noindent where $e_i$ is the $i$-th standard basis vector in $k^n$.

Write $(a_1,\cdots,a_n)^t$ instead of ${D_n}_{(1)}$ and $(b_1,\cdots,b_n)^t$ instead of ${D_n}_{(2)}$  and
put $c_i=T {u_i}_{x_{i+1}}$, for $2\leq i\leq n-1$. So $a_1=1+T{u}_x$, $a_i=T{u_i}_x$, for $2\leq i\leq n$, $b_1=T u_y$,
$b_2=1+T {u_2}_y$ and $b_i=T{u_i}_y$ for $3\leq i\leq n$.

\medskip

\begin{lemma} \em Let $d_n:=det\, D_n$. Then
$$d_n=a_1b_2-a_2b_1+\sum_{k=2}^{n-1} (-c_2)\cdots(-c_k) (a_1b_{k+1}-b_1a_{k+1})$$
\end{lemma}
\medskip

\begin{proof} Using the Laplace expansion of $d_n$ along the $n$-th column of $D_n$ we get
$$d_n=d_{n-1}+(-c_{n-1})det\, A_{n-1}$$
\noindent where $A_{n-1}$ is the $(n-1)\times (n-1)$ matrix obtained from $D_n$ by deleting the $(n-1)$-th row and $n$-th column.
One easily verifies that $det\, A_{n-1}=(-c_2)\cdots (-c_{n-2})(a_1b_n-b_1a_n)$. So
$$d_n=d_{n-1}+(-c_2)\cdots (-c_{n-1})(a_1b_n-b_1 a_n)$$
The result now follows by induction on $n$.
\end{proof}

\medskip

\noindent{\bf Proof (finished).} An easy calculation gives that
$$a_1b_2-a_2b_1=1+T(u_x+{u_2}_u)+T^2(u_x{u_2}_y-u_y{u_2}_x)$$
\noindent and if $2\leq k\leq n-1$, then
$$(-c_2)\cdots (-c_k)(a_1b_{k+1}-b_1 a_{k+1})=$$
$$(-1)^k{u_2}_{x_3}\cdots {u_k}_{x_{k+1}}(T^k{u_{k+1}}_y+T^{k+1}
(u_x{u_{k+1}}_y-u_y{u_{k+1}}_x))$$
\noindent Using the previous lemma,  it is left to the reader to deduce that, apart from a minus sign, the coefficient of $T^k$ in $d_n$
gives the $k$-th equation of Proposition \ref{nilpotency}, which concludes the proof.

\begin{corollary}\label{cor1}
Notations as in Proposition \ref{nilpotency}. If ${u_2}_{x_3}=0$, then $JH$ is nilpotent if and only if there exist
$\la_1,\la_2,c_1,c_2\in k$ and $f(T)\in k[T]$ such that $u=\la_2f(\la_1x+\la_2y)+c_1$ and $u_2=-\la_1f(\la_1x+\la2y)+c_2$.
\end{corollary}

\begin{proof} By Proposition \ref{nilpotency} we get that $JH$ is nilpotent if and only if $u_x+{u_2}_y=0$ and
$u_x{u_2}_y-u_y{u_2}_x=0$. Since ${u_2}_{x_3}=0$ the result follows from Theorem 7.2.25 \cite{vE}.
\end{proof}

So from now on {\em we may assume that} ${u_2}_{x_3}\neq 0$. Since ${u_n}_{x_{n+1}}=0$, there exists
$3\leq r\leq n$ such that ${u_i}_{x_{i+1}}\neq 0$ for all $2\leq i\leq r-1$ and ${u_r}_{x_{r+1}}=0$.
 By  Proposition \ref{nilpotency}, we have the following equations

$$u_x+{u_2}_y=0,$$
$$u_x{u_2}_y-u_y{u_2}_x = {u_2}_{x_3}{u_3}_y,$$
$${u_2}_{x_3}(u_x{u_3}_y-u_y{u_3}_x-{u_3}_{x_4}{u_4}_y)=0,$$
$$\vdots$$
$${u_2}_{x_3} \cdots {u_{r-2}}_{x_{r-1}}(u_x {u_{r-1}}_{y} - u_y {u_{r-1}}_{x}-{u_{r-1}}_{x_{r}}{u_{r}}_{y})=0,$$
$${u_2}_{x_3}\cdots {u_{r-1}}_{x_r}(u_x{u_r}_y-u_y{u_r}_x)=0.$$

\noindent Since ${u_2}_{x_3} \neq 0, \ldots, {u_{r-1}}_{x_r} \neq 0$, these equations become

$$u_x+{u_2}_y=0\,\,\,\,\,\,\,\,\,\,\,\,\,\,\,\,\,\, \,\,\,\,\,\,\,\,\,\,\,\,\,\,\,\,\,\,\,\,\,\,\,\,\,\,\,\,\,\,\,\,\,\,\,\,\,(1),$$
$$u_x{u_2}_y-u_y{u_2}_x = {u_2}_{x_3}{u_3}_y\,\,\,\,\,\,\,\,\,\,\,\,\,\,\,\,\,\,\,\,\,\,(2),$$
$$u_x{u_3}_y-u_y{u_3}_x = {u_3}_{x_4}{u_4}_y\,\,\,\,\,\,\,\,\,\,\,\,\,\,\,\,\,\,\,\,\,\,\,(3),$$
$$\vdots$$
$$u_x {u_{r-1}}_{y} - u_y {u_{r-1}}_{x}={u_{r-1}}_{x_{r}}{u_{r}}_{y},\,\,\,\,\mbox { (r-1)},$$
$$u_x{u_r}_y-u_y{u_r}_x = 0  \,\,\,\,\,\,\,\,\,\,\,\,\,\,\,\,\,\,\,\,\,\,\,\,\,\,\,\,\,\,\,\,\,\,\,\,\,\,\,\,\,\,\,\,\mbox{ (r)}.$$

\begin{corollary}\label{cor2} Let ${u_2}_{x_3}\neq 0$ and $r$ as above. If $u_y=0$, then $JH$ is nilpotent
if and only if $u\in k$ and ${u_i}_y=0$ for all $2\leq i\leq r$.
\end{corollary}

\begin{proof} The if-part follows from the equations $(1)\cdots (r)$. Conversely, assume that the equations  $(1)\cdots (r)$ hold.
Since $u_y=0$ equation $(r)$ gives $u_x{u_r}_y=0$. Assume $u_x\neq 0$. Then ${u_r}_y=0$. So equation $(r-1)$ implies
that ${u_{r-1}}_y=0$. Continuing in this way we arrive at ${u_3}_y=0$ and then by $(2)$ that ${u_2}_y=0$. This
contradicts equation $(1)$, since by assumption $u_x\neq 0$. Consequently $u_x=0$, i.e.\@ $u\in k$. It follows from $(1)$ that ${u_2}_y=0$ and that equation $(r)$ is satisfied. Furthermore, for each $2\leq i\leq r-1$ equation $(i)$ becomes
${u_i}_{x_{i+1}}{u_{i+1}}_y=0$, from which the desired result follows.
\end{proof}

\section{A lemma of  Dan Yan}

The following result was proved by  Dan Yan (see \cite[Lemma 2.1]{Yan}) for the case that the field $k$ is algebraically closed. We will extend her result
to arbitrary fields of characteristic zero. To keep this paper self-contained we give a short proof.

\medskip

 \begin{lemma}
\label{DanYan}
Let $k$ be a field of characteristic zero, $q\in k[x,y]$  and $0\neq w(q)\in k[q]$ such that $q_y|w^{e_1}{q_x}^{e_2}$ for some $e_1,e_2\geq 1$. If $p\nmid q_y$ for every $p\in k[x]\backslash k$, then
$q=P(y+b(x))$, for some $P(T)\in k[T]$ and $b(x)\in k[x]$.
\end{lemma}

\medskip

Let $p\in k[x,y]$ be irreducible. If $0\neq a\in k[x,y]$
we denote by $v_p(a)$ the number of factors $p$ in $a$. So $v_p(a)\geq 0$ and one easily verifies that if $a,b\in k[x,y]\backslash\{0\}$, then $v_p(ab)=v_p(a)+v_p(b)$. If $p_y\neq 0$, then $p\nmid p_y$ (look at degrees). One easily deduces

\medskip

\begin{equation}
\label{* lemma 2}
{\rm{If}} \, \, p_y\neq 0 \, \, {\rm{and}} \, \, d:=v_p(g)\geq 1, \, \, {\rm{then}} \, \, v_p(g_y)=d-1.
\end{equation}
 \medskip

\begin{proof} First assume that $k$ is algebraically closed.

\noindent  i) We show that $q_y|q_x$: let $p$ be irreducible and $v_p(q_y)=e\geq 1$. Then $p_y\neq 0$, for if $p_y=0$, then $p\in k[x]\backslash k$ divides $q_y$, contradicting the hypothesis. Also by the hypothesis $p|q_x$ or $p|w(q)$. We prove that in both cases
 $p^e|q_x$. Since this holds for all prime factors $p$ of $q_y$ we get $q_y|q_x$.\\\\
 \noindent Case 1. $p|q_x$. Then $d:=v_p(q_x)\geq 1$. So by (\ref{* lemma 2}) $v_p(q_{xy})=d-1$. Since $v_p(q_y)=e$ we get $v_p(q_{xy}\geq e-1$. So $d\geq e$, whence $p^e|q_x$.\\\\
 \noindent Case 2. $p|w(q)$. Since $k$ is algebraically closed we can write $w(q)$ as a product of factors $q+c$, with $c\in k$.
 So $p|q+c$, for some $c\in k$. Then $d:=v_p(q+c)\geq 1$. So by (\ref{* lemma 2}) $e=v_p(q_y)=d-1$, i.e.\@ $d=e+1$.
 Hence $p^{e+1}|q+c$. So $p^e|q_x$.\\
 \noindent  ii) Let $r:=deg_y q$. Then $r\geq 1$. Since $deg_y\,q_x\leq deg_y\, q_y +1$, it follows from $q_y|q_x$ that $q_x=(c_1(x)y+c_0(x))q_y$, for some $c_i\in k[x]$. The coefficient of $y^r$ gives $q'_r(x)=c_1(x)r q_r(x)$. Hence
 $deg_x\, q_r(x)=0$, i.e.\@ $q_r\in k^*$. So $0=c_1(x)rq_r$, whence $c_1(x)=0$. So $q_x=c_0(x)q_y$, i.e.\@
 $(\p_x-c_0(x)\p_y)q=0$. Let $b'(x)=c_0(x)$. Then $q\in k[y+b(x)]$, as desired.\\\\
\noindent iii) Now let $k$ be an arbitrary field of characteristic zero. From linear algebra
one knows that if $k\subseteq L$ is a field extension, then any system of non-homogeneous linear equations in $n$ variables with coefficients in $k$,
which has a solution in $L^n$, also has a solution in $k^n$. From this fact one readily deduces that if $a(x,y), b(x,y)\in k[x,y]$
are such that $b(x,y)|a(x,y)$ in $L[x,y]$, then also $b(x,y)|a(x,y)$ in $k[x,y]$.

 Finally assume that the hypothesis of Dan Yan's lemma are satisfied for polynomials in $k[x,y]$. Then they are
obviously satisfied in $\kb[x,y]$, where $\kb$ is an algebraic closure of $k$. It then follows from i)
that $q_y|q_x$ in $\kb[x,y]$. Hence, as observed above, $q_y|q_x$ in $k[x,y]$.  Then, by the argument given in ii),
 which does {\em not} use the algebraically closedness condition, we get the desired result.
 \end{proof}

\section{$u(x,y)=p(y+a(x))$}

In this section we {\em assume the relations of Proposition \ref{nilpotency} } and show that $u(x,y)=p(y+a(x))$ for some $a(x)\in k[x]$ and $p(T)\in k[T]$.

So we have the following situation: $n\geq 3$, $u=u(x,y)$, $u_i=u_i(x,y,x_{i+1})$ for all $2\leq i\leq n-1$ and
$u_n=h(x,y)$. We define $u_{n+1}=0$.
 Put
 $$D_0:=u_y\p_x-u_x\p_y$$
\noindent Then $k[x,y]^{D_0}=k[q]$
for some $q\in k[x,y]$ (see \cite[Theorem 1.2.25]{vE}). We may assume $q(0)=0$. The equations in Proposition \ref{nilpotency} can be written as
\begin{equation}
\label{eq 1}
u_x+{u_2}_y=0
\end{equation}

\begin{equation}
\label{eq 2}
-D_0(u_2)={u_2}_{x_3}{u_3}_y
\end{equation}

$${u_2}_{x_3}\cdots {u_{i-1}}_{x_i}(-D_0(u_i)-{u_i}_{x_{i+1}}{u_{i+1}}_y)=0, \mbox{ for all } 3\leq i\leq n$$
We may assume that $u_y\neq 0$ and ${u_2}_{x_3}\neq 0$.

\medskip

\begin{lemma}
\label{lemma 4.1}
Let $v=v_0(x,y)+\sum_{i=1}^d v_i(x)s^i$, with $v_d\neq 0$ and $d\geq 2$. If ${v_0}_y\neq 0$
and there exists $w\in k[x,y,t]$ such that

\begin{equation}
\label{4.1.1}
D_0(v)=-v_s w_y
\end{equation}
\noindent  then $v_d\in k^*$, $w_y=-\frac{1}{dv_d}v'_{d-1}(x)u_y$ and $v_y=Q(q)_y$ for some $Q(T)\in k[T]$ with $deg_T\,Q(T)\geq 1$.
\end{lemma}

\medskip

\begin{proof} The coefficient of $s^d$ in (\ref{4.1.1}) gives $v_d\in k^*$ and the coefficient of $s^{d-1}$ gives
$u_yv'_{d-1}(x)=-dv_d w_y$. So $w_y=-\frac{1}{dv_d} v'_{d-1}(x) u_y$. Then the coefficient of $s^0$ implies
that $D_0(v_0)=\frac{1}{dv_d} v'_{d-1}(x) v_1(x)u_y$. Let $b(x)\in k[x]$ with $b'(x))=\frac{1}{dv_d} v'_{d-1}(x) v_1(x)$.
Then $D_0(v_0)=D_0(b(x))$. So $v_0=b(x)+Q(q)$, for some $Q(T)\in k[T]$. So $v_y={v_0}_y=Q(q)_y$. Since
${v_0}_y\neq 0$ we get $deg_T\,Q(T)\geq 1$.
\end{proof}

\medskip


 Let $3\leq r\leq n$ be such that ${u_i}_{x_{i+1}}\neq 0$ for all $i<r$ and
${u_r}_{x_{r+1}}=0$ (observe that ${u_n}_{x_{n+1}}=h(x,y)_{x_{n+1}}=0$, so such an $r$ exists).

\medskip

\begin{proposition}
\label{u}
If $u$ and the $u_i$ satisfy the equations of Proposition \ref{nilpotency}, then $u=p(y+a(x))$, for
some $p(T)\in k[T]$ with $deg_T\,p(T)\geq 1$ and $a(x)\in k[x]$.
\end{proposition}

\medskip

\begin{proof} Let $r$ be as above. Then ${u_r}_{x_{r+1}}=0$ and ${u_2}_{x_3},\cdots,{u_{r-1}}_{x_r}$ are
all non-zero. So the above equations become

\begin{equation}
\label{eq 1 again}
u_x+{u_2}_y=0
\end{equation}

\begin{equation}
\label{eq i}
-D_0(u_i)={u_i}_{x_{i+1}}{u_{i+1}}_y, \,\mbox{ for all } 2\leq i\leq r-1
\end{equation}

\begin{equation}
\label{eq r}
D_0(u_r)=0.
\end{equation}

\noindent From $(\ref{eq r})$ we get $u_r=H(q)$, for some $H(T)\in k[T]$. Also $u=p(q)$. So $u_y=p'(q)q_y\equiv 0\,(mod\,q_y)$.
Since $-D_0(u_i)=u_x{u_i}_y-u_y{u_i}_x$ we get $-D_0(u_i)\equiv\,u_x {u_i}_y\,(mod\, q_y)$. So  by $(\ref{eq i})$ we get

\begin{equation}
\label{eq i^*}
u_x {u_i}_y\equiv\, {u_i}_{x_{i+1}}{u_{i+1}}_y\,(mod\,q_y), \mbox{ for all } 2\leq i\leq r-1.
\end{equation}

\medskip

\noindent Since $u_n=H(q)$ we get ${u_n}_y=H'(q)q_y\equiv 0\,(mod\, q_y)$. So by $(\ref{eq i^*})$ applied to $i = r-1$ we get
$u_x{u_{r-1}}_y\equiv 0\,(mod\,q_y)$. Then, multiplying $(\ref{eq i^*})$ ($i = r-2$) by $u_x$, we get
${u_x}^2 {u_{r-1}}_y\equiv 0\,(mod\, q_y)$. Continuing in this way we find that ${u_x}^{r-2} {u_2}_y\equiv 0\,(mod\, q_y)$.
Finally, $(\ref{eq 1})$ implies that ${u_x}^{r-1}\equiv 0\,(mod\,q_y)$. Since $u_x=p'(q)q_x$ we get that
$q_y|{p'(q)}^{r-1}{q_x}^{r-1}$. Let $d:=deg_y\, q$ and let $q_d(x)$ be the coefficient of $y^d$. In lemma \ref{lemma 4.4} below we will  show that $q_d(x)\in k^*$. So it follows from lemma \ref{DanYan}
that $q=p(y+a(x))$, for some $p(T)\in k[T]$ with $deg_T\,p(T)\geq 1$ and $a(x)\in k[x]$, which completes the proof.
\end{proof}


\medskip




In order to prove that $q_d\in k^*$ we need some preparations. By $\mathcal{T}\subseteq k[x,y]$ we denote the set
of {\em terms} $x^iy^j$ with $i,j\geq 0$. On $\T$ we define the {\em lexicographical ordering} $>$ as follows
$$x^{i_1}y^{j_1}>x^{i_2}y^{j_2}\mbox{ if } j_1>j_2 \mbox{ or, if } j_1=j_2 \mbox{ if } i_1>i_2$$
\noindent In other words, first look at the $y$-degree and in case of equality at the $x$-degree. This ordering
is a total ordering. If $0\neq f\in k[x,y]$ we can write $f$ as a finite sum of the form $f=\sum_{t\in\T} c_t t$, with
all $c_t\in k^*$. The greatest $t$ appearing in $f$ is called the {\em leading term of $f$}, denoted $lt(f)$. The corresponding
coefficient $c_t$ is called the {\em leading coefficient of $f$}, denoted $lc(f)$. The following easy result is crucial

\medskip

\begin{lemma}
\label{lemma 4.3}
Let $u,v\in k[x,y]$ with $lt(u)=x^{i_1}y^{j_1}$ and $lt(v)=x^{i_2}y^{j_2}$ be such that
$i_1,j_1\geq 1$, $i_2\geq 0$ and $j_2\geq 1$.Then
$$lt(u_xv_y-u_yv_x)=x^{i_1+i_1-1}y^{j_1+j_2-1}, \mbox{ if } i_1j_2-i_2j_1\neq 0$$
\end{lemma}

\begin{proof} The result follows easily from the fact that if $u=x^{i_1}y^{j_1}$ and $v=x^{i_2}y^{j_2}$ then
$(u_xv_y-u_yv_x)=(i_1j_2-i_2j_1)x^{i_1+i_1-1}y^{j_1+j_2-1}$.
\end{proof}

\medskip

\begin{lemma} $q_d\in k^*$.
\label{lemma 4.4}
\end{lemma}

\begin{proof}  i) Since $u_y\neq 0$ and $u=p(q)$ we get $q_y\neq 0$, so $d\geq 1$ and $N:=deg_T\,p(T)\geq 1$. We must  show that $s:=deg_x\,q_d(x)=0$. Therefore assume $s\geq 1$. We use the lexicographical order described above and
compute the leading terms of the $u_i$, for all $1\leq i\leq m+1$. First, from $u=p(q)$ it follows that $lt(u)=x^{sN}y^{dN}$. Then, by $(eq\,1)$ we get  $lt(u_2)=x^{sN-1}y^{dN+1}$.

First assume that $deg_{x_3}\, u_2\geq 2$. It then follows from lemma \ref{lemma 4.1} and $(\ref{eq 2})$ that ${u_2}_y=Q(q)_y$ for
some $Q(T)\in k[T]$ with $\rho:=deg_T\,Q(T)\geq 1$. So $lt({u_2}_y)=x^{\rho s}y^{\rho d-1}$. Consequently, $sN-1=\rho s$
and $dN+1=\rho d-1$. Multiplying the first equation by $d$, the second by $s$ and then subtracting these new equations we get  $-dm-s=s$, a contradiction. So we may assume that  $deg_{x_3}\, u_2=1$, i.e.\@ ${u_2}_{x_3}\in k^*$.
So there exists $2\leq m\leq n-1$, maximal such that $\la_2:={u_2}_{x_3}\in k^*,\cdots,
\la_m:={u_m}_{x_{m+1}}\in k^*$.  Observe $m\leq r-1$.  We claim that for all $2\leq i\leq m+1$ we have
$$lt(u_i)=x^{(i-1)sN-(i-1)}y^{(i-1)dN+1}$$
\noindent We use induction on $i$, the case $i=2$ is already done. So assume the case is proved for $i<m+1$. It follows from
$(\ref{eq i})$ that

\begin{equation}
\label{*1}
u_x{u_i}_y-u_y{u_i}_x={\la_i}{u_{i+1}}_y.
\end{equation}

\noindent It then follows from lemma \ref{lemma 4.3} that the leading term of the left hand side is equal to $x^{isN-i}y^{idN}$. Then
(\ref{*1}) gives that $lt(u_{i+1})=x^{isN-i}y^{idN+1}$, which proves the claim.\\
\noindent ii) In particular we have $lt(u_{m+1})=x^{msN-m}y^{mdN+1}$.  On the other hand, by lemma \ref{lemma 4.1}, there exists $Q(T)\in k[T]$ such that ${u_{m+1}}_y=Q(q)_y$. So if $deg_T Q(T)=\rho$, then  we get $lt({u_{m+1}}_y)=x^{\rho r} y^{\rho d -1}$.
Consequently $msN-m=\rho r$ and $mdN+1=\rho d-1$. Multiplying the first equation by $d$, the second by $s$ and then subtracting these new equations we get  $-dm-s=s$, a contradiction. So $s=0$, as desired.
\end{proof}

\section{The main result}

Now we will describe the main result of this paper. Recall that
\begin{equation}
\label{main}
H = (u(x,y),u_2(x,y,x_3), u_3(x,y,x_4), \ldots, u_{n-1}(x,y,x_{n}),u_n(x,y)).
\end{equation}
\noindent By Corollary \ref{cor1} and Corollary \ref{cor2}, in order to describe all $H$ in (\ref{main}) such that $JH$ is nilpotent,
we may assume that ${u_2}_{x_3}\neq 0$ and $u_y\neq 0$. As seen before, it follows from ${u_2}_{x_3}\neq 0$ that there
 exists $3\leq r\leq n$ such that ${u_i}_{x_{i+1}}\neq 0$ for all $2\leq i\leq r-1$ and ${u_r}_{x_{r+1}}=0$. Let $d_i:=deg_{x_{i+1}}\, u_i$,
 for all $2\leq i\leq n-1$. So $d_i\geq 1$ if $2\leq i\leq r-1$.

\begin{definition}
$P(T)\in k[T]$ of degree $d\geq 1$ is called {\em nice} if the coefficient of $T^{d-1}$ equals zero. The (leading) coefficient of $T^d$ will be denoted by $p_d$.
\end{definition}

\begin{theorem}
\label{main1} Let $H$ be as in \textnormal{(\ref{main})} with ${u_2}_{x_3}\neq 0$, $u_y\neq 0$ and $r$ as above. Then
 $JH$ is nilpotent if and only if the following conditions hold
 \begin{itemize}
 \item[(a)]
 $$u(x,y) = p(y + a(x)) \mbox{ and } u_2 = -a'(x)u  + P_2(x_3 + \frac{1}{d_2 p_{d_2}} b_2(x)),$$
for some $p(T)\in k[T]$ with $deg_T\,p(T)\geq 1, \, a(x), b_2(x) \in k[x]$ and $P_2(T)\in k[T]$ nice of degree $d_2$.
If $d_2\geq 2$, then $a''(x)=0$.
\vspace{0.2 cm}


\item[(b)] If $3 \leq i \leq r-1$ and $u_{i-1} = \sum_{j=1}^{l} c_{i-1,j}(x) u^j + P_{i-1}(x_{i}+\frac{1}{d_{i-1} p_{d_{i-1}}} b_{i-1}(x))$, with $c_{i-1,j}(x), b_{i-1}(x)\in k[x]$ and $P_{i-1}(T)$ nice of degree $d_{i-1}$, then
 $$u_i = -\frac{1}{d_{i-1} p_{d_{i-1}}} \big[\sum_{j=1}^l\frac{1}{j+1} {c'_{i-1,j}(x)} u^{j+1}+b_{i-1}'(x) u\big]+P_i(x_{i+1}+\frac{1}{d_i p_{d_i}} b_i(x))$$
for some $b_i(x)\in k[x]$ and $P_i(T)\in k[T]$, nice of degree $d_i$. If  $d_{i-1}\geq 2$, then  $c'_{i-1,j}(x)=0$ for all $j$.

\vspace{0.2 cm}

\item[(c)] If $u_{r-1} = \sum_{j=1}^{l} c_{r-1,j}(x) u^j + P_{r-1}(x_{r}+\frac{1}{d_{r-1} p_{d_{r-1}}} b_{r-1}(x))$, with $c_{r-1,j}(x), b_{r-1}(x)\in k[x]$ and $P_{r-1}(T)$ nice of degree $d_{r-1}$, then
$$u_r(x,y) = -\frac{1}{d_{r-1} {p_d}_{r-1}} \big[\sum_{j=1}^l\frac{1}{j+1} {c'_{r-1,j}(x)} u^{j+1}+b_{r-1}'(x) u\big]+b_r,$$
 with  ${c'_{r-1,j}}\in k$ for all $j\geq 1$ and $b_r\in k$, $b_{r-1}' \in k$. If $d_{r-1}\geq 2$, then $c'_{r-1,j}=0$ for all $j$.

 \vspace{0.2 cm}

\item[(d)]  No extra conditions on $u_i$ if $i>r$.

 \end{itemize}
\end{theorem}

\medskip





To prove this theorem we need some preliminaries:

\medskip

\begin{theorem}
\label{teorema 5.2} Let $v=\sum_{i=1}^l c_i(x)u^i+P(s+\frac{1}{dp_d}b(x))$, with $P$ nice of degree $d\geq 1$
and $b(x)\in k[x]$. Let $v$ and $w$ satisfy

\begin{equation}
\label{5.2.1}
D_0(v)=-v_s w_y
\end{equation}

\begin{equation}
\label{5.2.2}
D_0(w)=-w_t g_y
\end{equation}
\noindent for some $w\in k[x,y,t]$ with $e:=deg_t\,w\geq 0$ and $g\in k[x,y,r]$.\\
\begin{itemize}
\item[i)] If $e=0$, then
$$w=-\frac{1}{dp_d} \big(\sum_{i=1}^l\frac{1}{i+1} c'_i(x)u^{i+1}+b'(x) u\big)+c(x)$$
\noindent with $b'(x), c(x)\in k$ and $c'_i(x)\in k$ for all $i$.
\item[ii)]  If $e\geq 1$ there exist $c(x)\in k[x]$ and $Q(T)\in k[T]$, nice of degree $e$,
with leading coefficient $q_e$ such that
$$w=-\frac{1}{dp_d} \big(\sum_{i=1}^l\frac{1}{i+1} c'_i(x)u^{i+1}+b'(x) u\big)+Q(t+\frac{1}{e q_e} c(x))$$
\item[iii)] Furthermore, if $d\geq 2$, then $c'_i=0$ for all $i$.
\end{itemize}
\end{theorem}

\medskip

\begin{proof} Write $v=v_0(x,y)+\sum_{i=1}^d v_i(x) s^i$ and $w=w_0(x,y)+W$, where $W=0$ if $e=0$ and $W=\sum_{i=1}^e w_i(x,y) t^i$, if $e\geq 1$.
Then $v_d=p_d\in k^*$, $v_y={v_0}_y$, $w_y={w_0}_y$ (by (\ref{5.2.1})) and $w_e\in k^*$ (by (\ref{5.2.2})), if $e\geq 1$.

First assume $d\geq 2$. Then $w_y=-\frac{1}{dv_d} v'_{d-1}(x) u_y$ (by lemma \ref{lemma 4.1}). So $w_0=-\frac{1}{dv_d}v'_{d-1}(x)u+c(x)$
for some $c(x)\in k[x]$.  Put $b(x)=v_{d-1}(x)$. So, if $e=0$, then  $w=-\frac{1}{dp_d} b'(x) u+c(x)$ and if $e=1$
then $w=-\frac{1}{dp_d} b'(x) u+c(x)+q_1t$, where $q_1:=w_1$. Substituting these formulas in (\ref{5.2.1}) we get
$u_y\sum_{i=1}^l c'_i(x)u^i=0$, which implies that all $c'_i=0$, since $u$ contains $y$. If $e=0$ , then $w_t=0$, so \ref{5.2.2} implies that $b'(x), c(x)\in k$. This proves the case $d\geq 2$, $e\leq 1$.

Now let $e\geq 2$. Then by lemma \ref{lemma 4.1}, applied to (\ref{5.2.2}), we get $g_y=-\frac{1}{ew_e} w'_{e-1}(x) u_y$.
Substituting this formula into (\ref{5.2.2}) we get
$$u_x(-\frac{1}{dv_d}v'_{d-1}(x) u_y)-u_y({w_0}_x+\p_x(W))=-\frac{1}{ew_e}w'_{e-1}(x) u_y\p_t(W)$$
\noindent Also, using the formula for $w_0$ obtained above, we have
$${w_0}_x=-\frac{1}{dv_d}v'_{d-1}(x) u_x+-\frac{1}{dv_d}v''_{d-1}(x) u+c'(x)$$
\noindent So, combining the last two formulas, we get
$$-u_y[-\frac{1}{dv_d}v''_{d-1}(x) u+c'(x)+\p_x(W)]=u_y[-\frac{1}{ew_e}w'_{e-1}(x)\p_t(W)]$$
\noindent Hence
$$\frac{1}{dv_d}v''_{d-1}(x)u-c'(x)=(\p_x-\frac{1}{ew_e}w'_{e-1}(x)\p_t)W\in k[x,t]$$
\noindent Since $u_y\neq 0$ we get $v''_{d-1}(x)=0$. So $(\p_x-\frac{1}{ew_e}w'_{e-1}(x)\p_t)(c(x)+W)=0$,
whence $W=-c(x)+Q(t+\frac{1}{ew_e}w_{e-1}(x))$, for some $Q(T)\in k[T]$. Since $w=w_0+W$ and $w_0=-\frac{1}{dv_d}v'_{d-1}(x) u +c(x)$
we get the desired formula for $w$, using that $v_{d-1}=b(x)$ and $v_d=p_d$ and observing that $Q(T)$ is nice of degree $e$.
The statement in iii) follows again from  (\ref{5.2.1}), using that $w_y=-\frac{1}{dv_d} v'_{d-1}(x) u_y$.

Now, assume $d=1$. So $v=\sum_{i=1}^l c_i(x) u^i+p_1 s+b(x)$. Using (\ref{5.2.1}) we get
$$-u_y(\sum_{i=1}^l c'_i(x) u^i+b'(x))=p_1w_y=p_1{w_0}_y$$
\noindent So

\begin{equation}
\label{5.2.3}
w_0=-\frac{1}{p_1}\big(\sum_{i=1}^l \frac{1}{i+1}c'_i(x)u^{i+1}+b'(x) u\big)+c(x),
\end{equation}
\noindent for some $c(x)\in k[x]$. So, if $e=0$, \ref{5.2.2} implies again that $b'(x),c(x)\in k$ and all $c'_i(x)\in k$. So this  case is done. Also the case $e=1$ is done, using that $w=w_0+q_1t$. So assume
that $e\geq 2$. Then, as observed above $g_y=-\frac{1}{ew_e} w'_{e-1}(x) u_y$. By (\ref{5.2.2}) and (\ref{5.2.3}) we get
$$(-\frac{1}{p_1})\big[\sum_{i=1}^l c''_i(x) u^{i+1}+b''(x)u\big]+c'(x)=-(\p_x-\frac{1}{ew_e}w'_{e-1}(x)\p_t)(W)\in k[x,t]$$
\noindent  Since $u$ contains $y$ we get that $b''(x)=0$ and all $c''_i(x)=0$. So
$$(\p_x-\frac{1}{ew_e}w'_{e-1}(x)\p_t)(W+c(x))=0$$
\noindent Hence $W=-c(x)+Q(t+\frac{1}{ew_e}w_{e-1}(x))$, for some $Q(T)\in k[T]$, which is nice of degree $e$. Then the formula
for $w$ follows from $w=w_0+W$ and (\ref{5.2.3}).

\end{proof}

Now we prove the main result of this paper

\medskip

\noindent{\bf{Proof of Theorem \ref{main1}}:} 
As seen above the proof of Corollary \ref{cor2}, the nilpotency of $JH$ is equivalent to the following equations
$$u_x+{u_2}_y=0\,\,\,\,\,\,\,\,\,\,\,\,\,\,\,\,\,\, \,\,\,\,\,\,\,\,\,\,\,\,\,\,\,\,\,\,\,\,\,\,\,\,\,\,\,\,\,\,\,\,\,\,\,\,\,(1),$$
$$u_x{u_2}_y-u_y{u_2}_x = {u_2}_{x_3}{u_3}_y\,\,\,\,\,\,\,\,\,\,\,\,\,\,\,\,\,\,\,\,\,\,(2),$$
$$u_x{u_3}_y-u_y{u_3}_x = {u_3}_{x_4}{u_4}_y\,\,\,\,\,\,\,\,\,\,\,\,\,\,\,\,\,\,\,\,\,\,\,(3),$$
$$\vdots$$
$$u_x {u_{r-1}}_{y} - u_y {u_{r-1}}_{x}={u_{r-1}}_{x_{r}}{u_{r}}_{y},\mbox { (r-1)},$$
$$u_x{u_r}_y-u_y{u_r}_x = 0  \,\,\,\,\,\,\,\,\,\,\,\,\,\,\,\,\,\,\,\,\,\,\,\,\,\,\,\,\,\,\,\,\,\,\,\,\,\,\,\,\mbox{ (r)}.$$

\noindent First assume that $JH$ is nilpotent. So to prove the theorem we need to solve the $r$ equations above.  Let $2\leq j\leq r-1$ and write
$$u_j = u_{j,0}(x,y) + \displaystyle \sum_{i=1}^{d_j} u_{j,i}(x,y)x_{j+1}^i$$





\noindent  As ${u_j}_{x_{j+1}} \neq 0,$  we obtain $d_j \geq 1$ and
if $i \geq 1$  it follows from  from (j) and $u_y\neq 0$ that $u_{j,i} = u_{j,i}(x)$. So  ${u_j}_y = {u_{j,0}}_y$. Moreover
we obtain from equation (j) that $u_{j,d_j}\in k^*$.

\begin{itemize}
\item[(a)] By Proposition \ref{u} we have that $u=p(y + a(x))$ for
some $p(T)\in k[T]$ with $deg_T\,p(T)\geq 1$ and $a(x)\in k[x]$.
From (1) we get $u_{2,0}=-a'(x)u+c(x)$, with $c(x)\in k[x]$. So if $d_2=1$, then $u_2$ has the desired form. If
$d_2\geq 2$, then $u_2=-a'(x)+c(x)+U_2$, where $U_2=\sum_{i=1}^{d_2} u_{2,i}(x)x_3^i$. It follows from (2) and
lemma \ref{lemma 4.1} that ${u_{3}}_y=-\frac{1}{d_2 p_{d_2}}b'_2(x){u_2}_y$, for some $b_2(x)\in k[x]$.
Substituting these formulas in (2), an easy calculation gives
$$a''(x)u-c'(x)=(\p_x-\frac{1}{d_2 p_{d_2}}b'_2(x)\p_{x_3})U_2\in k[x,x_3]$$
\noindent Since $u$ contains $y$ we get $a''(x)=0$ and hence
$$(\p_x-\frac{1}{d_2 p_{d_2}}b'_2(x)\p_{x_3})(U_2+c(x))=0$$
\noindent So $U_2=-c(x)+P_2(x_3+\frac{1}{d_2 p_{d_2}}b_2(x))$, for some $P_2(T)\in k[T]$, nice
of degree $d_2$. Since $u_2=-a'(x)+c(x)+U_2$ it follows that $u_2$ has the desired form.

\vspace{0.2 cm}

\item[(b)] This case follows directly from Theorem \ref{teorema 5.2} ii) and iii)

\vspace{0.2 cm}

\item[(c)] $u_r$ is obtained by using Theorem \ref{teorema 5.2} i).

\vspace{0.2 cm}

\item[(d)] This follows immediately from the equations $(1),\cdots,(r)$, which do not contain $u_i$ with $i>r$.
\end{itemize}

Conversely, it is left the reader to verify that the formulas obtained in $(a)\cdots (d)$ indeed satisfy the
equations $(1)\cdots (r)$, which shows that the corresponding $H$ has a nilpotent Jacobian matrix.

\section{Invertibility}

Throughout this section
$$H=(u(x,y),u_2(x,y,x_3),u_3(x,y,x_4),\cdots,u_{n-1}(x,y,x_n),u_n(x,y))$$
In the previous sections we completely described all such maps $H$ with the property  that $JH$ is nilpotent. For the the corresponding maps $F=X+H$ we have that
$det\,JF=1$. So if the Jacobian Conjecture is true, $F$ must be invertible. The main result of this section (Theorem \ref{th3} below)
confirms this. More precisely we show that $F$ is a product of elementary maps (see definition below), i.e.

\begin{theorem}\label{th3}
If $H$ is as above and $JH$ is nilpotent, then $F\in E(k,n)$.
\end{theorem}

Before we prove this result we make some preliminary remarks. Recall that a polynomial map is called {\em elementary}
if it is of the form $(x_1,\cdots,x_{i-1},x_i+a,x_{i+1},\cdots,x_n)$ for some $a\in k[x]$ not containing $x_i$. We denote such a map for short as $(x_i+a)$. The subgroup
of $Aut_k k[x_1,\cdots,x_n]$ generated by these elementary maps is denoted by $E(k,n)$. Two polynomial maps $F$ and $G$
are called {\em elementary equivalent} if there exist $E_1,E_2\in E(k,n)$ such that $G=E_1\circ F\circ E_2$. Since the $E_i$ are invertible we have that $F$ is invertible if and only if $G$ is invertible. So to prove Theorem \ref{th3} it suffices to
show that $F$ is elementary equivalent to the identity map.

First we consider the case ${u_2}_{x_3}=0$, described in Corollary \ref{cor1}.

\begin{proposition}\label{prop1} Notations as in Corollary \ref{cor1}. Then $F\in E(k,n)$.
\end{proposition}

\begin{proof} First let $n>3$. By the description given in Corollary \ref{cor1} we get
$$(F_1,F_2)=(x+\la_2f(\la_1x+\la_2 y)+c_1,y-\la_1f(\la_1x+\la_2y)+c_2$$
$$F_i=x_i+u_i(x,y,x_{i+1})\mbox{ for all } 3\leq i\leq n-1 \mbox{ and } F_n=x_n+u_n(x,y)$$
Let $T$ be the translation $(x-c_1,x-c_2,x_3,\cdots,x_n)$. Replacing $F$ by $T\circ F$ we may assume that $c_1=c_2=0$.
Furthermore we may assume that $\la_1=\la_2=0$: if for example $\la_1\neq 0$ let $S$ be the invertible linear map
$$(\la_1x+\la_2y,y,x_3,\cdots,x_n)$$
Then $S\circ F\circ S^{-1}=(x,y,F'_2,\cdots,F'_n)$, with $F'_i=x_i+\tilde{u}_i(x,y,x_{i+1})$ for all $3\leq i<n$
and $F'_n=x_n+\tilde{u}_n(x,y)$. So we may assume
$$F=(x,y,x_3+u_3(x,y,x_4),\cdots,x_{n-1}+u_{n-1}(x,y,x_n),x_n+u_n(x,y)$$
Finally, let  $E_n=(x,y,\cdots,x_{n-1},x_n-u_n(x,y))$. Then
$$E_n\circ F=(x,y,x_3+u_3(x,y,x_4),\cdots,x_{n-1}+u_{n-1}(x,y,x_n),x_n)$$
Now one readily verifies that this map belongs to $E(k,n)$, which implies the proposition in case $n>3$.
The case $n=3$ is left to the reader.
\end{proof}

Next we consider the case ${u_2}_{x_3}\neq 0$ and $u_y=0$, described in Corollary \ref{cor2}.

\begin{proposition}\label{prop2} Notations as in Corollary \ref{cor2}. Then $F\in E(k,n)$.
\end{proposition}

\begin{proof} By the description of Corollary \ref{cor2} we get
$$(F_1,\cdots,F_r)=(x+u,y+u_2(x,x_3),\cdots,x_{r-1}+u_{r-1}(x,x_r),x_r+u_r(x))$$
$F_n=x_n+u_n(x,y)$ and if there exists $r<i<n$, then $F_i=x_i+u_i(x,y,x_{i+1})$. Replacing $F$ by $(x-u)\circ F$ we may
assume that $F_1=x$. Then, replacing $F$ by $(x_r-u_r(x))\circ F$, we may assume that $u_r=0$. Next, replacing
$F$ by $(x_{r-1}-u_{r-1}(x,x_r))\circ F$, we may assume that $u_{r-1}=0$. Continuing in this way we arrive at
$(F_1,\cdots,F_r)=(x,y,x_3,\cdots,x_r)$. So if $r=n$ we are done. Now let $r<n$. Then consider $(x_n-u_n(x,y))\circ F$.
So we may assume that $u_n=0$. Next consider $(x_{n-1}-u_{n-1}(x,y,x_n))\circ F$ etcetera. Finally we arrive
at the identity map, which proves the proposition.
\end{proof}

So from Proposition \ref{prop1} and Proposition \ref{prop2} it follows, that in order to prove Theorem \ref{th3}, we
may assume from now on that $u_y\neq 0$ and ${u_2}_{x_3}\neq 0$ and that we have an $r$ as above.
First we claim $F$ is invertible if and only if $(F_1,\cdots,F_r)$ is invertible:  if $r=n$ there is nothing to prove, so assume $r<n$.
Using that $F_1,\cdots,F_r\in k[x_1,\cdots,x_r]$, $F_n=x_n+u_n(x,y)$
and $F_i=x_i+u_i(x,y,x_{i+1})$ for all $i>r$, it is an easy exercise to show that $F$ is elementary equivalent to the map $$(F_1,\cdots,F_r,x_{r+1},\cdots,x_n)$$
\noindent Furthermore, since the polynomials  $F_1,\cdots,F_r\in k[x_1,\cdots,x_r]$ it is well-known that $(F_1,\cdots,F_r,x_{r+1},\cdots,x_n)$
is invertible if and only if $(F_1,\cdots,F_r)$ is. This implies our claim. So it suffices to show that $(F_1,\cdots,F_r)\in E(k,r)$.


Using the notations of Theorem \ref{main1} we introduce some new notations. First, if $2\leq i<r$ let $l_i$
denote the coefficient of $T^{d_i}$ in $P_i(T)$ and $L_i:=d_il_i$. Furthermore, put $d_1:=2$, $L_1:=1$, $l_r=0$
and $b_1(x):=a(x)$. Let $s\geq 2$ be maximal such that $d_{s-1}\geq 2$. So $2\leq s\leq r$ and $d_i=1$ if $s\leq i<r$.
Hence $L_i=l_i$ if $s\leq i<r$. Finally define
$$\ga_{k,t}:=L_{s-1+(t-1)}^{-1}\cdots L_{s-1+(t-k)}^{-1}, \mbox{ for all  } 1\leq k\leq t\leq r-s+1$$
\noindent One readily verifies that
$$\ga_{1,t}=L_{s-1+t-1}^{-1}\mbox{ and }\ga_{k,t-1}=L_{s-1+t-1}\ga_{k+1,t}, \mbox{ if } 1\leq k\leq t-1\,\,\,\,(*)$$
\noindent Then the next result follows by induction on $t$, using Theorem \ref{main1} and (*).

\begin{proposition}\label{prop6.1}
If $1\leq t\leq r-s+1$, then
$$u_{s-1+t}=\sum_{k=1}^t (-1)^k\frac{1}{k!}\ga_{k,t} b_{s-1+t-k}^{(k)}(x)u^k+l_{s-1+t}x_{s+t}+b_{s-1+t}(x)$$
\noindent with $b_{s-1}^{(r-s+2)}=\cdots=b_{r-1}^{(2)}=b_r^{(1)}=0$.
\end{proposition}

\begin{corollary}\label{cor6.2}
Let $F=(x+u,x_2+u_2,\cdots,x_r+u_r)$. Then for every $1\leq t\leq r-s$ there exists $E_t\in E(k,r)$ such that
$F\circ E_t=(F_1,\cdots,F_{r-t-1},\widetilde{F}_{r-t},\widetilde{F}_{r-t+1},\cdots,\widetilde{F}_r)$
\noindent where $\widetilde{F}_{r-i}=x_{r-i}+b_{r-i}(F_1)+l_{r-i}x_{r-i+1}$, for all $0\leq i<t$ and
$$\widetilde{F}_{r-t}=\sum_{k=1}^{r-s-t+1}(-1)^k\ga_{k,r-s-t+1}\big[\frac{1}{k!}b_{r-t-k}^{(k)}u^k+\frac{1}{(k+1)!}b_{r-t-k}^{(k+1)}u^{(k+1)}+\cdots+\frac{1}{(k+t)!}b_{r-t-k}^{(k+t)}u^{k+t}\big]$$
$$+b_{r-t}(F_1)+l_{r-t}x_{r-t+1}+x_{r-t}$$
\end{corollary}

\begin{proof} By induction on $t$. First the case $t=1$. From Proposition \ref{prop6.1} (with $t=r-s+1$) and $l_r=0$ we get
$F_r=x_r+[u_r]+b_r$, where
$$[u_r]:=\sum_{k=1}^{r-s+1} (-1)^k\frac{1}{k!}\ga_{k,r-s+1}b^{(k)}_{r-k}(x)u^k$$
\noindent with $b_r\in k$ and $b^{(k+1)}_{r-k}(x)=0$ for all $1\leq k\leq r-s+1$. From  Proposition \ref{prop6.1} (with $t=r-s$)   we get
$$F_{r-1}=x_{r-1}+\sum_{k=1}^{r-s} (-1)^k\frac{1}{k!}\ga_{k,r-s}b^{(k)}_{r-1-k}(x)u^k+l_{r-1}x_r+b_{r-1}(x)$$
\noindent Define $E_1=(x_1,\cdots,x_{r-1},x_r-[u_r])$. Observe that $[u_r]\in k[x,x_2]$ and $r>2$. So $E_1\in E(k,r)$.
Furthermore $F\circ E_1=(F_1,\cdots,F_{r-2},\wF_{r-1},x_r+b_r)$, where
$$\wF_{r-1}=x_{r-1}+\sum_{k=1}^{r-s} (-1)^k\frac{1}{k!}\ga_{k,r-s}b^{(k)}_{r-1-k}(x)u^k+l_{r-1}x_r$$
$$+\sum_{k=1}^{r-s+1} (-1)^{k+1}\frac{1}{k!}l_{r-1}\ga_{k,r-s+1}b^{(k)}_{r-k}(x)u^k+b_{r-1}(x)$$
\noindent Now write
$$\sum_{k=1}^{r-s+1} (-1)^{k+1}\frac{1}{k!}l_{r-1}\ga_{k,r-s+1}b^{(k)}_{r-k}(x)u^k=l_{r-1}\ga_{1,r-s+1}b^{(1)}_{r-1}u$$
$$+\sum_{k=1}^{r-s} (-1)^k\frac{1}{(k+1)!}l_{r-1}\ga_{k+1,r-s+1}b^{(k+1)}_{r-1-k}(x)u^{k+1}$$
\noindent and use that $l_{r-1}\ga_{k+1,r-s+1}=\ga_{k,r-s}$ and $l_{r-1}\ga_{1,r-s+1}=1$. Then we get
$$\wF_{r-1}=x_{r-1}+\sum_{k=1}^{r-s} (-1)^k\ga_{k,r-s}\big[\frac{1}{k!}b^{(k)}_{r-1-k}(x)u^k+\frac{1}{(k+1)!}b^{(k+1)}_{r-1-k}(x)u^{k+1}\big]$$
$$+l_{r-1}x_r+b^{(1)}_{r-1}(x)u+b_{r-1}(x)$$
\noindent Since by Proposition \ref{prop6.1} $b^{(2)}_{r-1}(x)=0$, it follows from Taylor's theorem that
$b_{r-1}(F_1)=b_{r-1}(x+u)=b_{r-1}(x)+b^{(1)}_{r-1}(x)u$. This finishes the proof of the case $t=1$

Now assume $t\geq 1$ and that we already know the existence of a map $E_t$, having the properties as described
in the statement of this corollary. In particular we have
$\wF_{r-t}=x_{r-t}+[u_{r-t}]+b_{r-t}(F_1)+l_{r-t}x_{r-t+1}$. Observe that $[u_{r-t}]\in k[x,x_2]$ and define
$$E':=(x_1,\cdots,x_{r-t-1},x_{r-t}-[u_{r-t}],x_{r-t+1},\cdots,x_r)$$
\noindent Then a similar argument as given for the case $t=1$ above, shows that $(F\circ E_t)\circ E'$ has the desired
form.

\end{proof}


\begin{corollary}\label{cor6.3}
Let $F=(x+u,x_2+u_2,\cdots,x_r+u_r)$. Then
$F$ is elementary equivalent to
$(F_1,\cdots,F_{s-1},\widetilde{F}_{s},x_{s+1},\cdots,x_r)$, where $\wF_{s}=x_{s}+L_{s-1}^{ -1}b_{s-1}(x)$.
\end{corollary}

\begin{proof}
By Corollary \ref{cor6.2}, with $t=r-s$, there exists $E\in E(k,r)$ such that
$$F\circ E=(F_1,\cdots,F_{s-1},\wF_s, x_{s+1}+b_{s+1}(F_1)+l_{s+1}x_{s+2},\cdots,x_{r-1}+b_{r-1}(F_1)+l_{r-1}x_r,x_r)$$
\noindent where
$$\wF_s=x_s-L_{s-1}^{-1}\big[b^{(1)}_{s-1}(x)u+\frac{1}{2!}b^{(2)}_{s-1}(x)u^2+\cdots+\frac{1}{(r-s+1)!}b^{(r-s+1)}_{s-1}(x)u^{r-s+1}\big]$$
$$+b_s(F_1)+l_s x_{s+1}$$
\noindent Since $b^{(r-s+2)}_{s-1}(x)=0$, by Proposition \ref{prop6.1},  it follows from Taylor's theorem, using $F_1=x+u$,
that
$$b_{s-1}(F_1)=b_{s-1}(x)+b^{(1)}(x)u+\frac{1}{2!}b^{(2)}_{s-1}(x)u^2+\cdots\frac{1}{(r-s+1)!}b^{(r-s+1)}_{s-1}(x)u^{r-s+1}$$
\noindent So
$$\wF_s=x_s-L_{s-1}^{-1}\big[b_{s-1}(F_1)-b_{s-1}(x)\big]+b_s(F_1)+l_sx_{s+1}$$

\noindent So if we define
$$E':=(x_1,\cdots,x_{s-1},x_s+L_{s-1}^{-1}b_{s-1}(x_1)-b_s(x_1),x_{s+1}-b_{s+1}(x_1),\cdots,x_{r-1}-b_{r-1}(x_1),x_r)$$
\noindent Then $E'\in E(k,r)$ and
$$E'\circ F\circ E=(F_1,\cdots,F_{s-1},x_s+L_{s-1}^{-1}b_{s-1}(x)+l_s x_{s+1},x_{s+1}+l_{s+1}x_{s+2},\cdots,x_{r-1}+l_{r-1}x_r,x_r)$$
\noindent One readily verfies that $E'\circ F\circ E$ is elementary equivalent to
$$F':=(F_1,\cdots,F_{s-1},x_s+L_{s-1}^{-1}b_{s-1}(x),x_{s+1},\cdots,x_r)$$
\noindent which completes the proof.
\end{proof}
Now we are ready to prove

\begin{proposition}\label{prop6.4} Let $F=(x+u,x_2+u_2,\cdots,x_r+u_r)$. Then $F\in E(k,r)$.
\end{proposition}

\begin{proof} We use induction on $n(H)$:= the number of $d_i\geq 2$. Since $d_1=2$ we have $n(H)\geq 1$.
First the case $n(H)=1$. So $s=2$. It follows from Corollary \ref{cor6.3} that $F$ is elementary equivalent
to $(F_1,\wF_2,x_3,\cdots,x_r)$, where $\wF_2=x_2+a(x)$ ($L_1$=1 and $b_1(x)=a(x)$). Since
$F_1=x+p(x_2+a(x))$, the case $n(H)=1$ follows.

So let $n(H)>1$. Then $s\geq 3$. Since $d_{s-1}\geq 2$ it follows from Theorem \ref{main1} that $u_{s-1}=[u_{s-1}]+P_{s-1}(x_s+L_{s-1}^{-1}b_{s-1}(x))$, where $[u_{s-1}]=\sum c_{s-1,j} u^j$, with $c_{s-1,j}\in k$ for all $j$. So by
Corollary \ref{cor6.3} $F$ is elementary equivalent to
$$F':=(F_1,\cdots,F_{s-2},x_{s-1}+[u_{s-1}]+P_{s-1}(x_s+L_{s-1}^{-1}b_{s-1}(x)),x_s+L_{s-1}^{-1}b_{s-1}(x),x_{s+1},\cdots,x_r)$$
\noindent Now define the elementary map
$$E'':=(x_1,\cdots,x_{s-1},x_s-L_{s-1}^{-1}b_{s-1}(x),x_{s+1},\cdots,x_r)$$
\noindent Then
$$F'\circ E''=(F_1,\cdots,F_{s-2},x_{s-1}+[u_{s-1}]+P_{s-1}(x_s),x_{s},\cdots,x_r)$$
\noindent Consequently, $F'\circ E''$ is elementary equivalent to $(F_1,\cdots,F_{s-2},x_{s-1}+[u_{s-1}],x_s,\cdots,x_r)$.
Finally put $\wH:=(u_1,\cdots,u_{s-2},[u_{s-1}],0,\cdots,0)$. Then obviously $\wH$ is special and $n(\wH)=n(H)-1$.
It follows from Proposition \ref{nilpotency} that $J(\wH)$ is nilpotent. So by the  induction hypothesis we get that
$F'\circ E''\in E(k,r)$, which implies that $F\in E(k,r)$, as desired.
\end{proof}

\end{document}